\documentclass{amsart} 
\usepackage[utf8]{inputenc}
\usepackage{amssymb}
\usepackage{amsmath}
\usepackage{amsthm}
\usepackage{mathtools}
\usepackage{multirow}
\usepackage{thmtools}
\usepackage{geometry} 
\usepackage{graphicx}
\usepackage{epstopdf}
\usepackage{amsfonts}
\usepackage{tikz}
\usepackage{enumitem}

\usepackage{url}
\usepackage[colorlinks=true, citecolor=blue, backref=page]{hyperref}
			
\usepackage{hyperref} 			
\usepackage{color} 

\newtheorem{theorem}{Theorem}[section]

\newtheorem{lemma}{Lemma}[section]
\newtheorem{corollary}{Corollary}[section]
\newtheorem{remark}{Remark}[section]
\newtheorem{question}{Question}[section]

\begin{document}
	
	
	\title[The minimal covolume hyperbolic right-angled Coxeter group]{The minimal covolume right-angled Coxeter group \\ in hyperbolic 3-space}
	
	\author[ A.~Egorov, A.~Vesnin]{A.~Egorov, A.~Vesnin} 
\address{Sobolev Institute of Mathematics, Siberian Branch of Russian Academy of Sciences, Novosibirsk, Russia}
	\email{a.egorov2@g.nsu.ru} 
	\email{vesnin@math.nsc.ru}
	
\begin{abstract}
		We prove that among all right-angled Coxeter groups in hyperbolic 3-space, the group generated by reflections in the faces of a right-angled triangular bipyramid with three ideal and two finite vertices has the smallest covolume. The group is arithmetic and its covolume equals Catalan's constant $G = 0.915965\ldots$.  
	\end{abstract} 
	
	\keywords{hyperbolic polyhedron, right-angled polyhedron, right-angled Coxeter group} 
	\subjclass[2000]{51F15, 57K32}	
	
	\thanks{The authors were supported by the state task of Sobolev Institute of Mathematics. Namely, A.V. was supported by the project No.~FWNF-2022-0004 and A.E. was supported by the project No.~FWNF-2022-0017. \\ This is a preprint of the work accepted for publication in Siberian Mathematical Journal \textcopyright \, 2025 Pleiades Publishing, Ltd. https://www.pleiades.online/}
	
	\maketitle
	
	\section{Introduction} \label{sec1} 
	
	A fundamental problem in hyperbolic geometry is the study of discrete subgroups of the group $\operatorname{Isom}(\mathbb H^n)$ of isometries of $n$-dimensional hyperbolic space $\mathbb H^n$, in particular,  groups generated by reflections. At the same time, discrete torsion-free isometry groups correspond to hyperbolic $n$-dimensional manifolds. In many constructions, such groups arise as finite-index subgroups of reflection groups~\cite{Bes, Ve1, Rat}.    
	
Recall~\cite{Cox,Vi1} that a Coxeter group $W$ is defined by a finite presentation of the form 
	$W = \langle s \in S \, | \, (s t)^{m_{st}} = 1, \, \forall s, t \in S \rangle$, 
	where $m_{ss} = 1$ and $m_{st} \in \{2, 3, \ldots, \infty\}$ if $s \neq t$. Here $m_{st} = \infty$ means that there are no relations between $s$ and $t$. A Coxeter group $W$ is called right-angled if $m_{st} \in \{2,\infty\}$ for $s\neq t$. 
	
	A convex polyhedron $P \subset \mathbb H^n$ with dihedral angles of the form $\pi/m$ for integer $m \geq 2$ at $(n-2)$-dimensional faces is called a hyperbolic Coxeter polyhedron. The group $\Gamma(P)$, generated by reflections in the $(n-1)$-dimensional faces of $P$, is a Coxeter group. The covolume of $\Gamma(P)$ is defined as the volume $\operatorname{vol} (P)$. We say that $\Gamma(P)$ is cocompact if $P$ is a compact polyhedron, and that $\Gamma(P)$ has finite covolume if $P$ has finite volume. 
It was shown by Vinberg~\cite{Vi3}, if $n \geq 30$, then no cocompact Coxeter groups exist in $\mathbb H^n$. Examples are known only for $d \leq 8$~\cite{Bug}. According to~\cite{Pro, Kho}, if $n > 995$, then no Coxeter groups of finite covolume exist in $\mathbb H^n$. Examples are known only for $n \leq 19$~\cite{ViKa} and $n=21$~\cite{Bor}. As shown in~\cite{All}, there are infinitely many Coxeter groups of finite covolume (resp. cocompact) in $\mathbb H^n$ for each $n \leq 19$ (resp. $n \leq 6$).
	
This work considers right-angled polyhedra of finite volume in three-dimensional hyperbolic space $\mathbb H^3$ and their corresponding right-angled Coxeter groups. A polyhedron $P$ is called right-angled if all its dihedral angles equal $\pi/2$. In this case, the corresponding reflection group $\Gamma(P)$ is a right-angled Coxeter group. It is known that no compact right-angled hyperbolic Coxeter groups exist if $n>4$~\cite{PoVi}, and none with finite-volume fundamental polyhedron if $n>12$~\cite{Duf}. Examples are known in dimensions $n \leq 8$, see~\cite{Vi2}. 

Dunbar and Meyerhoff~\cite{DuMe} showed that the set of volumes of three-dimensional hyperbolic orbifolds of finite volume has ordinal type $\omega^{\omega}$, and the number of orbifolds of a given volume is finite. 
Traditionally, volumes of polyhedra in three-dimensional hyperbolic space are computed using the following \textit{Lobachevsky function}, see~\cite{Mil}, 
$$
\Lambda (\theta) = - \int\limits_0^{\theta} \log | 2 \sin (t) | \, {\rm d} t.  \label{lob}
$$
Below we will use the value $v_{oct} = 8 \Lambda(\pi/4) = 3.663862$, equal to the volume of a regular ideal octahedron in $\mathbb H^3$, and the value $v_{tet} = 3 \Lambda(\pi/3) = 1.014941$, equal to the volume of a regular ideal tetrahedron in $\mathbb H^3$. Here and throughout, all approximate values of the Lobachevsky function and volume values are given to six decimal places.  
	
Discrete reflection groups are conveniently described using Coxeter diagrams~\cite{Vi1, Vi4}. To each Coxeter polyhedron, particularly in $\mathbb H^3$, one can corresponds a graph called its Coxeter diagram. It is defined as follows. The vertices of the Coxeter diagram correspond to the faces of the polyhedron. If two faces are perpendicular, the vertices are not connected by an edge. If the angle between faces is $\pi/m$, $m\geq 3$, the corresponding vertices are connected by an edge of multiplicity $m-2$, if $m \in \{ 3,4,5\}$, or a regular edge with label $m$. Coxeter diagrams are also used to denote Coxeter groups generated by reflections in the faces of a Coxeter polyhedron. 
	
Let $\Delta_{3,4,4}$ denote a tetrahedron in $\mathbb H^3$ with faces $f_1, f_2, f_3, f_4$ where the dihedral angles $\alpha_i$ between faces $f_i$ and $f_{i+1}$, $i=1,2,3$, are $\alpha_1 = \pi/3$, $\alpha_2= \pi/4$, $\alpha_3 = \pi/4$, and all other dihedral angles equal $\pi/2$. The Coxeter diagram for the group $\Gamma(\Delta_{3,4,4})$, generated by reflections in the faces of $\Delta_{3,4,4}$, is shown in Fig.~\ref{figCo}\,(a), with face notations indicated.
\begin{figure}[ht]	
\begin{center} 
\begin{tikzpicture}[scale=0.7] 				
\draw[very thick, black] (0,0) -- (1,0);
\draw[very thick, black] (1,-0.1) -- (2,-0.1);
\draw[very thick, black] (1,0.1) -- (2,0.1);
\draw[very thick, black] (2,-0.1) -- (3,-0.1);
\draw[very thick, black] (2,0.1) -- (3,0.1);
\fill[black] (0,0) circle (4pt); 
\fill[black] (1,0) circle (4pt);
\fill[black] (2,0) circle (4pt);
\fill[black] (3,0) circle (4pt); 
\node[] at (0,0.6) {$f_1$};
\node[] at (1,0.6) {$f_2$};
\node[] at (2,0.6) {$f_3$};
\node[] at (3,0.6) {$f_4$};
\node[] at (1.5,-1.2) {$(a)$};
\end{tikzpicture}
\qquad \qquad 
\begin{tikzpicture}[scale=0.8] 				
\draw[very thick, black] (1.9,0) -- (1.9,1);
\draw[very thick, black] (2.1,0) -- (2.1,1);
\draw[very thick, black] (2,0.1) -- (3,-0.6);
\draw[very thick, black] (2,-0.1) -- (3,-0.8);
\draw[very thick, black] (2,0.1) -- (1,-0.6);
\draw[very thick, black] (2,-0.1) -- (1,-0.8);
\fill[black] (2,1) circle (4pt); 
\fill[black] (2,0) circle (4pt);
\fill[black] (1,-0.7) circle (4pt);
\fill[black] (3,-0.7) circle (4pt); 
\node[] at (2.,-1.2) {$(b)$};
\end{tikzpicture}
\end{center} 
\caption{Coxeter diagrams for groups $\Gamma(\Delta_{3,4,4})$ and $\Gamma(\Delta'_{3,4,4})$.} \label{figCo} 	
\end{figure}
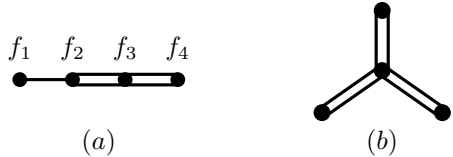  

The tetrahedron $\Delta_{3,4,4}$ has three finite vertices and one ideal vertex belonging to faces $f_2$, $f_3$, and $f_4$. Under the action of the dihedral group of order six, generated by reflections in faces $f_1$ and $f_2$, six copies of $\Delta_{3,4,4}$ form the tetrahedron $\Delta'_{3,4,4}$ in Fig.~\ref{fig32}\,(a), whose three ideal vertices lie in a one plane containing face $f_4$ and three right angles meet in the finite vertex (such tetrahedra are called trirectangular in~\cite{AbSt}). The Coxeter diagram of the groups $\Gamma(\Delta'_{3,44})$ generated by reflections in faces of $\Delta'_{3,4,4}$ is presented in Fig.~\ref{figCo}\,(b). 
Combining $\Delta'_{3,4,4}$ with its mirror image across the plane containing the face $f_4$ yields a triangular bipyramid with six faces and all dihedral angles $\pi/2$. Since this bipyramid has three ideal and two finite vertices, we will denote it by   $\mathcal P_{(3,2)}$. The polyhedron $P_{(3,2)}$ and its Schlegel diagram are shown in Fig.~\ref{fig32}\,(b) and (c). Note that $\mathcal P_{(3,2)}$ has appeared in various contexts in papers~\cite{ERT, RT, PoVi, Pro}. 
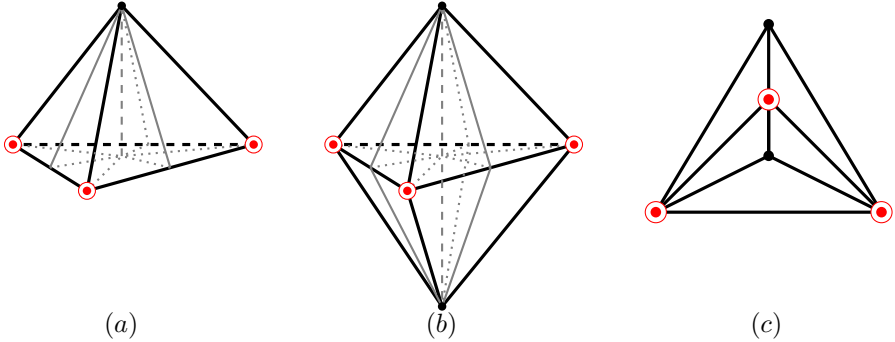
\begin{figure}[ht]	
\begin{center} 
	\begin{tikzpicture}[scale=0.8] 		
			\coordinate (A) at (0,0,2);
			\coordinate (B) at (-2,0,0);
			\coordinate (C) at (2,0,0);
			\coordinate (V1) at (0,2.5,0.5);
			\coordinate (V2) at (0,-2.5,0.5);
			\coordinate (AB) at (-1,0,1);
			\coordinate (BC) at (0.25,0,0);
			\coordinate (CA) at (1,0,1);
			\coordinate (O) at (0,0,0.5);
			\draw[very thick, black] (A) -- (B);
			\draw[very thick, black, dashed] (C) -- (B);
			\draw[very thick, black] (C) -- (A);
			\draw[very thick, black] (A) -- (V1);
			\draw[very thick, black] (B) -- (V1); 
			\draw[very thick, black] (C) -- (V1);	
			\draw[thick, gray] (AB) -- (V1);
			\draw[thick, gray, dotted] (BC) -- (V1);
			\draw[thick, gray] (CA) -- (V1);
			\draw[thick, gray, dashed] (V1) -- (O);
			\draw[thick, gray, dotted] (A) -- (O);
			\draw[thick, gray, dotted] (B) -- (O);
			\draw[thick, gray, dotted] (C) -- (O);
			\draw[thick, gray, dotted] (AB) -- (O);
			\draw[thick, gray, dotted] (BC) -- (O);
			\draw[thick, gray, dotted] (CA) -- (O);
			\foreach \v in {V1} {
				\fill[black] (\v) circle (2pt);}
				\foreach \v in {V2} {
				\fill[white] (\v) circle (2pt);}
			\foreach \v in {A, B, C} {
				\fill[white] (\v) circle (4pt); 
				\draw[red] (\v) circle (4pt); 
				\fill[red] (\v) circle (2pt);}
		\node[] at (-0.2,-3) {$(a)$};
		\end{tikzpicture}
	\qquad 
		\begin{tikzpicture}[scale=0.8] 		
			\coordinate (A) at (0,0,2);
			\coordinate (B) at (-2,0,0);
			\coordinate (C) at (2,0,0);
			\coordinate (V1) at (0,2.5,0.5);
			\coordinate (V2) at (0,-2.5,0.5);
			\coordinate (AB) at (-1,0,1);
			\coordinate (BC) at (0.25,0,0);
			\coordinate (CA) at (1,0,1);
			\coordinate (O) at (0,0,0.5);
			\draw[very thick, black] (A) -- (B);
			\draw[very thick, black, dashed] (C) -- (B);
			\draw[very thick, black] (C) -- (A);
			\draw[very thick, black] (A) -- (V1);
			\draw[very thick, black] (B) -- (V1); 
			\draw[very thick, black] (C) -- (V1);	
			\draw[very thick, black] (A) -- (V2);
			\draw[very thick, black] (B) -- (V2); 
			\draw[very thick, black] (C) -- (V2);
			\draw[thick, gray] (AB) -- (V1);
			\draw[thick, gray, dotted] (BC) -- (V1);
			\draw[thick, gray] (CA) -- (V1);
			\draw[thick, gray] (AB) -- (V2);
			\draw[thick, gray, dotted] (BC) -- (V2);
			\draw[thick, gray] (CA) -- (V2);
			\draw[thick, gray, dashed] (V1) -- (V2);
			\draw[thick, gray, dotted] (A) -- (O);
			\draw[thick, gray, dotted] (B) -- (O);
			\draw[thick, gray, dotted] (C) -- (O);
			\draw[thick, gray, dotted] (AB) -- (O);
			\draw[thick, gray, dotted] (BC) -- (O);
			\draw[thick, gray, dotted] (CA) -- (O);
			\foreach \v in {V1, V2} {
				\fill[black] (\v) circle (2pt);}
			\foreach \v in {A, B, C} {
				\fill[white] (\v) circle (4pt); 
				\draw[red] (\v) circle (4pt); 
				\fill[red] (\v) circle (2pt);}
		\node[] at (-0.2,-3) {$(b)$};
		\end{tikzpicture}
\qquad 
\begin{tikzpicture}[scale=1.] 				
\coordinate (A) at (0,0);
\coordinate (B) at (3,0);
\coordinate (C) at (1.5, 2.5);
\coordinate (D) at (1.5, 0.75);
\coordinate (E) at (1.5, 1.5);
\draw[very thick, black] (A) -- (B) -- (C) -- cycle;
\draw[very thick, black] (A) -- (D) -- (B) -- (E) -- cycle;
\draw[very thick, black] (C) -- (D);
\draw[very thick, black] (D) -- (E);
\foreach \v in {C, D} {
\fill[black] (\v) circle (2pt);
}
\foreach \v in {A,B,E} {
\fill[white] (\v) circle (4pt); 
\draw[red] (\v) circle (4pt); 
\fill[red] (\v) circle (2pt);
}
\fill[white] (0,-1) circle (2pt);
		\node[] at (1.47,-1.5) {$(c)$};
\end{tikzpicture}
\end{center} 
\caption{Tetrahedron $\Delta'_{3,4,4}$, polyhedron $\mathcal P_{(3,2)}$ and Schlegel diagram of $\mathcal P_{(3,2)}$.} \label{fig32} 	
\end{figure}

By construction, volume of the right-angled polyhedron $\mathcal P_{(3,2)}$ equals $\operatorname{vol} (\mathcal P_{(3,2)}) = 2 \operatorname{vol} (\Delta'_{3,4,4}) = 12 \operatorname{vol} (\Delta_{3,4,4}) = 2 \Lambda\left(\frac{\pi}{4}\right)$, where the volume of $\Delta_{3,4,4}$ is computed via the Lobachevsky function using formula (\ref{eqnKel}) below. It is well-known~\cite{OEIS} that $2 \Lambda\left(\frac{\pi}{4}\right) = G$, where 
$$
G = \sum_{n=0}^{\infty} \frac{(-1)^n}{(2n+1)^2}
$$ 
is Catalan's constant, introduced in his work~\cite{Cat} in 1867. To six decimal places, $G = 0.915965$. More precise approximations of Catalan's constant $G$ can be found in~\cite{Pap}. 

The main result of the present paper is the following 

\begin{theorem} \label{theorem1.1}
Let $\mathcal P$ be a right-angled hyperbolic polyhedron in $\mathbb H^3$. Then the inequality $\operatorname{vol}(P) \geq G$ holds, where $G = 2 \Lambda\left(\frac{\pi}{4}\right)$ is Catalan's constant. Moreover, the triangular bipyramid $P_{(3,2)}$ is the unique right-angled polyhedra for which the equality holds. 
\end{theorem}

It is well-known that the arithmeticity of groups of three-dimensional hyperbolic manifolds and orbifolds plays an important role in studying of their properties~\cite{MaRe}. The question that comes to Siegel~\cite{Sie} is: which hyperbolic manifolds and orbifolds have the smallest volume in the orientable and non-orientable cases? As noted in~\cite{Bel}, there is a \textit{folklore conjecture} that the minimal volumes are always achieved by arithmetic manifolds or orbifolds. By now, this conjecture has been fully confirmed for $n=3$, see~\cite{Ada, ChFr, CFJR, GMM, GeMa, MaMa, Mey}. Note that a similar property holds for right-angled groups. Namely, the minimal cocompact right-angled hyperbolic Coxeter group is arithmetic by~\cite{AMR} and~\cite{BoDu}. As a corollary of Theorem~\ref{theorem1.1}, the minimal covolume right-angled hyperbolic Coxeter group is arithmetic also.   

\begin{corollary} \label{cor1}
The right-angled hyperbolic Coxeter group in $\mathbb H^3$ of minimal covolume is arithmetic.  
\end{corollary}

The paper as organized as follows. In Section~\ref{sec2}, we recall some results about right-angled polyhedra in $\mathbb H^3$ and their volumes. More detailed information about the geometry of $\mathbb H^3$ and  hyperbolic manifolds and orbifolds can be found in~\cite{Rat}. In Section~\ref{sec3}, we present the proof of Theorem~\ref{theorem1.1}, structured as a sequence of Lemmas~\ref{lemma3.1}--\ref{lemma3.4}. In Section~\ref{sec4}, we discuss the arithmeticity of the right-angled reflection groups introduced in Section~\ref{sec3}. We conclude the paper with some open questions formulated in Section~\ref{sec5}. 

\section{Preliminaries} \label{sec2}

\subsection{Existence of right-angled hyperbolic polyhedra} 
Let $\mathbb R^{n, 1}$ denote the vector space $\mathbb R^{n + 1}$ equipped with a scalar product $\langle \cdot, \cdot \rangle$ of signature $(n, 1)$, and let $f_n$ be the associated quadratic form. In an appropriate basis, this form can be expressed as:  
$$
f_n (x) =  -x_0^2 + x_1^2 + \dotsm + x_n^2.
$$
The \emph{Lobachevsky space $\mathbb H^n$ of dimension $n$} is defined as the upper connected component of the hyperboloid given by the equation $f_n(x) = -1$:
$$
\mathbb H^{n} = \{ x \in \mathbb R^{n, 1} \mid f_n (x) = -1\, \text{ and }\, x_{0} > 0 \}.
$$
In this model, points at infinity correspond to isotropic vectors:
$$
\partial \mathbb H^n = \{x \in \mathbb R^{n, 1} \mid f_n (x) = 0\, \text{ and }\, x_{0} > 0\} / \mathbb R_+.
$$

A \emph{convex hyperbolic polyhedron} of dimension $n$ is the intersection of a finite family of closed half-spaces in $\mathbb H^n$ that contains a non-empty open set. A convex hyperbolic polyhedron is called a \emph{Coxeter hyperbolic polyhedron} if all its dihedral angles are integer fractions of $\pi$, i.e., of the form $\pi/m$ for some integer $m \geq 2$. A Coxeter hyperbolic polyhedron is called \emph{right-angled} if all its dihedral angles equal $\pi/2$. If all dihedral angles of a generalized\footnote{A \emph{generalized convex polyhedron} $P$ is an intersection (with non-empty interior), possibly of infinitely many half-spaces in $\mathbb H^n$, such that every closed ball intersects only finitely many boundary hyperplanes defining $P$.} polyhedron do not exceed $\pi/2$, the polyhedron is said to be \emph{acute-angled}.

It is known that generalized Coxeter polyhedra are natural fundamental domains for discrete groups generated by reflections in spaces of constant curvature (see~\cite{Vi4}).

A convex $n$-dimensional polyhedron has finite volume if and only if it is the convex hull of finitely many points in the compactification $\overline{\mathbb H^n} = \mathbb H^n \cup \partial \mathbb H^n$. An $n$-dimensional polyhedron is compact if and only if it is the convex hull of finitely many points in $\mathbb H^n$, which are called \emph{finite}. A convex polyhedron is called \emph{ideal} if all its vertices lie on the absolute $\partial \mathbb H^n$ (such vertices are called \emph{ideal}). It is known~\cite[Th.~1]{An1} that for an acute-angled finite-volume polyhedron $\mathcal P \subset \mathbb H^3$ any finite vertex is incident to three faces and any infinite vertex is incident to three or four faces. 

Two polyhedra $P$ and $P'$ in Euclidean space $\mathbb E^n$ are said to be \emph{combinatorially equivalent} if there exists a bijection between their sets of faces that preserves incidence relations. The class of combinatorially equivalent polyhedra is called a \emph{combinatorial type} of polyhedron. Note that if a hyperbolic polyhedron $P \subset \mathbb H^n$ has finite volume, then its closure $\overline{P} \subset \overline{\mathbb H^n}$ is combinatorially equivalent to some compact polyhedron in $\mathbb E^n$.

The following theorem is a special case of Andreev's theorems for the compact case~\cite{An1} and the finite-volume case~\cite{An2}, see also~\cite{RHD}. Andreev's theorems provide necessary and sufficient conditions for realizing an abstract polyhedron of given combinatorial type and prescribed dihedral angles in Lobachevsky space. We present these conditions for right-angled polyhedra, following~\cite[Th.~2.1]{Atk}. Let $P^*$ denote the planar graph dual to the 1-skeleton $P^{(1)}$ of the polyhedron $P$. 

\begin{theorem}  \cite{An1, An2} \label{theorem2.1}
An abstract polyhedron $P$ can be realized as a right-angled polyhedron $\mathcal{P}$ in $\mathbb H^3$ if and only if the following conditions hold:
\begin{itemize}
	\item[(1)] $P$ has at least six faces.
	\item[(2)] Each vertex of $P$ has degree three or four.
	\item[(3)] For any triple of faces $(F_i, F_j, F_k)$ such that $F_i \cap F_j$ and $F_j \cap F_k$ are edges of $P$ with distinct endpoints, $F_i \cap F_k = \emptyset$ holds.
	\item[(4)] The dual graph $P^*$ contains no prismatic $k$-circuits for $k \leq 4$.
\end{itemize}
Moreover, each degree-three vertex in $P$ corresponds to a finite vertex in $\mathcal{P}$, each degree-four vertex in $P$ corresponds to an ideal vertex in $\mathcal{P}$, and the realization is unique up to isometry.
\end{theorem}	

Here, for a planar graph $G$ and its dual graph $G^*$, a \textit{$k$-circuit} is a simple closed curve composed of $k$ edges in $G^*$. A \textit{prismatic $k$-circuit} is a $k$-circuit $\gamma$ such that no two edges of $G$ corresponding to edges traversed by $\gamma$ share a common vertex.

\subsection{Volume of a birectangular hyperbolic tetrahedron} A tetrahedron in $\mathbb H^3$ is called birectangular (or an orthoscheme) if its vertices can be labeled as $A, B, C, D$ such that edge $AB$ is orthogonal to face $BCD$, and face $ABC$ is orthogonal to edge $CD$. In this case, the following dihedral angles are equal: $\angle AC= \angle BC= \angle BD = \pi/2$. The remaining dihedral angles are denoted by $\angle AB = \alpha$, $\angle AD = \beta$, $\angle CD = \gamma$, where $\alpha + \beta \geq \pi/2$ and $\beta + \gamma \geq \pi/2$. Such a birectangular tetrahedron is denoted by $R(\alpha, \beta, \gamma)$. A formula for its volume was derived in~\cite{Ke1}: 
\begin{equation}
\begin{gathered}
	\operatorname{vol} (R(\alpha, \beta, \gamma)) = \frac{1}{2} \left[ \Lambda(\alpha + \delta) + \Lambda(\alpha - \delta) + \Lambda \left( \frac{\pi}{2} + \beta - \delta\right) + \Lambda \left( \frac{\pi}{2} - \beta + \delta \right) \right.  \\ 
	\left. + \Lambda (\gamma + \delta) - \Lambda(\gamma - \delta) + 2 \Lambda \left( \frac{\pi}{2} - \delta \right) \right], 
\end{gathered} \label{eqnKel}
\end{equation}
where
$
0 \leq \delta = \arctan \frac{\sqrt{\cos^2 \beta - \sin^2 \alpha \sin^2 \gamma}}{\cos \alpha \, \cos \gamma} < \frac{\pi}{2}.
$

Using formula~(\ref{eqnKel}), we compute the covolume of the group $\Gamma(\Delta_{3,4,4})$, whose Coxeter diagram is shown in Fig.~\ref{figCo}\,(a), and the group $\Gamma(\Delta_{4,4,4})$, whose Coxeter diagram is shown in Fig.~\ref{figCo2}. Namely, since $\Delta_{3,4,4} = R(\pi/3, \pi/4, \pi/4)$, we have $\operatorname{vol} (\Delta_{3,4,4}) = \frac{1}{6} \Lambda(\pi/4)$, and similarly, since $\Delta_{4,4,4} = R(\pi/4, \pi/4, \pi/4)$, we have $\operatorname{vol} (\Delta_{4,4,4}) = \frac{1}{2} \Lambda(\pi/4)$.  

\subsection{Compact right-angled polyhedra}
Since the conditions for realizing a combinatorial polyhedron as a compact right-angled polyhedron in $\mathbb H^3$ were first formulated by Pogorelov in~\cite{Pog}, these polyhedra are sometimes called \textit{Pogorelov polyhedra}. 

Let us describe an important infinite family of compact right-angled polyhedra. For $n \geq 5$, consider the $(2n+2)$-hedron $L_n$, whose top and bottom bases are $n$-gons, and whose lateral surface consists of two cycles of $n$ pentagons~\cite{Ve1}, in particular, $L_5$ is a dodecahedron, see Fig.~\ref{figL5}\,(a). By Theorem~\ref{theorem2.1}, $L_n$ can be realized in $\mathbb H^3$ as a compact right-angled polyhedron $\mathcal L_n$. Following~\cite{Ve1}, the polyhedra $\mathcal L_n$ are called \textit{L{\"o}bell polyhedra}, and the three-dimensional hyperbolic manifolds corresponding to torsion-free subgroups of index eight in $\Gamma(\mathcal L_n)$, $n \geq 5$, are called \textit{L{\"o}bell manifolds}, see~\cite{Ve4}. 
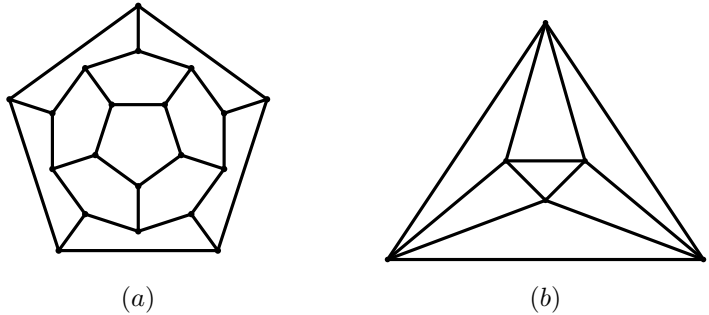
\begin{figure}[ht]	
	\begin{center} 
\begin{tikzpicture}[scale=0.6, rotate=180] 
   \coordinate (O) at (0,0);
	\foreach \angle [count=\i] in {234,306,...,594} {
		\coordinate (V\i) at (\angle:1);}
	\draw[very thick, black] (V1) -- (V2) -- (V3) -- (V4) -- (V5) -- cycle;
	\foreach \angle [count=\i] in {18,54,...,378} {
		\coordinate (U\i) at (\angle:2);}
	\draw[very thick, black] (U1) -- (U2) -- (U3) -- (U4) -- (U5) -- (U6) 
	-- (U7) -- (U8) -- (U9) -- (U10) -- cycle;
	\foreach \angle [count=\i] in {54,126,...,414} {
		\coordinate (W\i) at (\angle:3);}
	\draw[very thick, black] (W1) -- (W2) -- (W3) -- (W4) -- (W5) -- cycle;
	\draw[very thick, black] (V2) -- (U9);
	\draw[very thick, black] (V1) -- (U7);
	\draw[very thick, black] (V3) -- (U1);
	\draw[very thick, black] (V4) -- (U3);
	\draw[very thick, black] (V5) -- (U5);
	\draw[very thick, black] (W4) -- (U8);
	\draw[very thick, black] (W1) -- (U2);
	\draw[very thick, black] (W2) -- (U4);
	\draw[very thick, black] (W3) -- (U6);
	\draw[very thick, black] (W5) -- (U10);	
	\foreach \v in {V1,V2,V3,V4,V5,W1,W2,W3,W4,W5,U1,U2,U3,U4,U5,U6,U7,U8,U9,U10} {
		\fill[black] (\v) circle (2pt);}
		\node[] at (0.,3.5) {$(a)$};
\end{tikzpicture}
\qquad \qquad 
\begin{tikzpicture} [scale=0.525]
	\draw[very thick, black] (-3.00,-0.50) -- (1.00,5.50);
	\draw[very thick, black] (1.00,5.50) -- (5.00,-0.50);
	\draw[very thick, black] (5.00,-0.50) -- (-3.00,-0.50);
	\draw[very thick, black] (-3.00,-0.50) -- (0.00,2.0);
	\draw[very thick, black] (0.00,2.0) -- (1.00,5.50);
	\draw[very thick, black] (2.00,2.0) -- (1.00,5.50);
	\draw[very thick, black] (2.00,2.0) -- (5.00,-0.50);
	\draw[very thick, black] (-3.00,-0.50) -- (1.00,1.0);
	\draw[very thick, black] (1.00,1.0) -- (5.00,-0.50);
	\draw[very thick, black] (1.00,1.0) -- (2.00,2.00);
	\draw[very thick, black] (2.00,2.00) -- (0.00,2.0);
	\draw[very thick, black] (1.00,1.00) -- (0.00,2.0);
	\fill[black] (1.00,5.5) circle (2pt); 
	\fill[black] (-3.00,-0.50) circle (2pt);
	\fill[black] (5.00,-0.50) circle (2pt);
	\fill[black] (2.00,2.0) circle (2pt);
	\fill[black] (0.00,2.00) circle (2pt);
	\fill[black] (1.00,1.00) circle (2pt);	
	\node[] at (1.,-1.5) {$(b)$};
\end{tikzpicture}
	\end{center} \caption{Dodecahedron $L_5$ and octahedron $A_3$.} \label{figL5} 	
\end{figure} 

\begin{theorem}~\cite[Cor.~9.2]{In1}, \label{th3}
The compact right-angled hyperbolic polyhedron of minimal volume is the dodecahedron $\mathcal L_5$, and the next smallest polyhedron is $\mathcal L_6$. 
\end{theorem}

Below is a formula expressing the volumes of right-angled hyperbolic polyhedra $\mathcal L_n$ in terms of the Lobachevsky function. 

\begin{theorem}~\cite{Ve2} For $n \geq 5$, the following equality holds: 
$$
\operatorname{vol} (\mathcal L_n) = \frac{n}{2} \left( 2 \Lambda(\theta_n) + \Lambda \left(\theta_n + \frac{\pi}{n}\right) + \Lambda \left(\theta_n - \frac{\pi}{n}\right) - \Lambda \left( 2 \theta_n - \frac{\pi}{2} \right) \right),
$$
where $\theta_n = \frac{\pi}{2} - \arccos \left( \frac{1}{2 \cos (\pi/n)} \right)$.
\end{theorem}

Approximate values are $\operatorname{vol} (\mathcal L_5)=4.306207$ and $\operatorname{vol} (\mathcal L_6) = 6.023046$. It is easy to see that  $\operatorname{vol}(\mathcal L_n)$ is an increasing function of $n$, see~\cite[Th.~4.2]{In1}, and $\lim_{n\to \infty} \frac{\operatorname{vol} (\mathcal L_n)}{n} = \frac{5}{4} v_{tet}$, see~\cite[Prop~2.10]{MPVe}. 
The paper~\cite{In2} lists the first $825$ volumes of compact right-angled hyperbolic polyhedra, along with images of the first hundred corresponding polyhedra. Volume computations were performed using the computer program Orb~\cite{Hea}.  

Upper and lower bounds for the volumes of compact right-angled polyhedra in terms of their number of vertices were obtained by Atkinson in~\cite{Atk}.
\begin{theorem} \cite[Th.~2.3]{Atk} \label{th2.4}
Let $\mathcal P$ be a compact right-angled hyperbolic polyhedron with $V$ vertices. Then
\begin{equation}
	\frac{v_{oct}}{32} (V-8) \leq \operatorname{vol}(\mathcal P) < \frac{5v_{tet}}{8} (V - 10). \label{eqn2}
\end{equation}
Moreover, there exists a sequence of compact right-angled polyhedra $\mathcal P_i$ with $V_i$ vertices such that $\operatorname{vol}(\mathcal P_i) / V_i$ tends to $\frac{5}{8} v_{tet}$ as $i$ tends to infinity. 
\end{theorem}

Note that in virtue of Theorem~\ref{theorem2.1} it is assumed in Theorem~\ref{th2.4} that $V \geq 20$. In~\cite{EgVe}, the upper bound (\ref{eqn2}) was improved for compact right-angled hyperbolic polyhedra with $V \geq 24$ vertices, and in~\cite{ABEV}, for those with $V > 80$.

\subsection{Ideal right-angled polyhedra} A polyhedron in $\mathbb H^3$ is called \textit{ideal} if all its vertices are ideal. 

Let us describe an important family of ideal right-angled polyhedra. For $n \geq 3$ consider a $(2n+2)$-hedron with top and bottom $n$-gonal bases and a lateral surface consisting of two layers of $n$ triangles, where four edges meet at each vertex. We call such a polyhedron \textit{$n$-antiprism} and denote by $A_n$. Note that $A_3$ is an octahedron, see Fig.~\ref{figL5}\,(b).

By Theorem~\ref{theorem2.1}, for any $n \geq 3$ the polyhedron $A_n$ can be realized in $\mathbb H^3$ as an ideal right-angled polyhedron $\mathcal A_n$. It is shown in~\cite[Prop.~5]{Kol} that if a polyhedron has the minimal number of faces among all ideal right-angled polyhedra in $\mathbb H^3$ with at least one $n$-gonal face, then it is the antiprism $\mathcal A_n$.

Below is a formula expressing volumes of polyhedra $\mathcal A_n$ in terms of the Lobachevsky function. 

\begin{theorem}~\cite{Thu} For $n \geq 3$, the following equality holds: 
\begin{equation}
	\operatorname{vol} (\mathcal A_n) = 2n \left[ \Lambda \left( \frac{\pi}{4} + \frac{\pi}{2n} \right) + \Lambda \left( \frac{\pi}{4} - \frac{\pi}{2n} \right) 
	\right]. \label{eqnantiprism}
\end{equation}
\end{theorem}

Upper and lower bounds for the volumes of ideal right-angled polyhedra in terms of their number of vertices were obtained by Atkinson in~\cite{Atk}.
\begin{theorem} \cite[Th.~2.2]{Atk} \label{th5}
Let $\mathcal P$ be an ideal right-angled hyperbolic polyhedron with $V$ vertices. Then
\begin{equation}
	\frac{v_{oct}}{4} (V-2) \leq \operatorname{vol}(\mathcal P) < \frac{v_{oct}}{2} (V - 4). \label{eqnAtkIdeal}
\end{equation}
Both inequalities become equalities if $\mathcal P$ is a regular ideal hyperbolic octahedron. Moreover, there exists a sequence of ideal right-angled polyhedra $\mathcal P_i$ with $V_i$ vertices such that $\operatorname{vol}(\mathcal P_i) / V_i$ tends to $\frac{1}{2} v_{oct}$ as $i$ tends to infinity. 
\end{theorem}

Note that in virtue Theorem~\ref{theorem2.1} it is assumed in Theorem~\ref{th5} that $V \geq 6$. In~\cite{EgVe}, the upper bound in (\ref{eqnAtkIdeal}) was improved for ideal right-angled hyperbolic polyhedra with $V \geq 8$ vertices, and in~\cite{ABEV}, for those with $V > 24$.

\subsection{Right-angled polyhedra with finite and ideal vertices}

Suppose a right-angled hyperbolic polyhedron $\mathcal P$ has $V_f$ finite and $V_\infty$ ideal vertices. Let $E$ denote its number of edges, and $F$ its number of faces. The Euler characteristic $\chi(\mathcal P)$ of $\mathcal P$ is 
$$
\chi(\mathcal P) = V_\infty + V_f - E + F = 2.
$$ 
Since each finite vertex to three ideal vertex, and each ideal vertex is incident to four edges, we have $3 V_f + 4 V_\infty = 2E$. Hence, 
\begin{equation}
F = V_\infty +  \frac{1}{2} V_f +  2, \label{eqn3}
\end{equation}
which implies that the number $V_f$ of finite vertices is always even.
Given that by condition (1) of Theorem~\ref{theorem2.1}, $F \geq 6$, we obtain 
\begin{equation}
V_{\infty} + \frac{1}{2} V_f  \geq 4. \label{eqn4}
\end{equation}

\begin{lemma} \label{rem}
Let $f$ be a face of a right-angled polyhedron $\mathcal P \subset \mathbb H^3$. If $f$ is triangular, then it contains at least two ideal vertices, and if $f$ is quadrilateral, then it contains at least one ideal vertex.
\end{lemma}

\begin{proof}
Recall that the sum of the interior angles $\alpha_1, \ldots, \alpha_n$ of an $n$-gon in $\mathbb H^2$ satisfies:
$ \sum_{i=1}^n \alpha_i < (n-2)\pi$, 
where in finite vertices of face $f$, the interior angle is $\pi/2$, and in ideal vertices, it is $0$. 
If $f$ is a triangular face with $k$ finite vertices, then $k \cdot \frac{\pi}{2} < \pi$, so $k \leq 1$. If $f$ is a quadrilateral face with $k$ finite vertices, then $k\cdot\frac{\pi}{2} < 2 \pi$, so $k \leq 3$. 
\end{proof} 

Atkinson~\cite{Atk} established the following upper and lower bounds for the volume of a right-angled hyperbolic polyhedron with at least one ideal vertex. 

\begin{theorem} \cite[Th.~2.4]{Atk} \label{th7}
Let $\mathcal P$ be a right-angled hyperbolic polyhedron with $V_\infty \geq 1$ ideal and $V_F$ finite vertices. Then the following inequalities hold: 
\begin{equation}
\frac{v_{oct}}{8} \cdot V_\infty + \frac{v_{oct}}{32} \cdot V_f - \frac{v_{oct}}{4} \leq \operatorname{vol}(\mathcal P) <  \frac{v_{oct}}{2} \cdot V_\infty + \frac{5 v_{tet}}{8} \cdot V_f - \frac{v_{oct}}{2}. \label{eqn7}
\end{equation}
\end{theorem}

In~\cite{ABEV}, the upper bound in (\ref{eqn7}) was improved for right-angled hyperbolic polyhedra with $V_{\infty} \geq 1$ and $V_{\infty} + V_{F} \geq  18$. Noting that $v_{oct} = 4G$, we rewrite the lower bound from~(\ref{eqn7}) as:   
\begin{equation}
\operatorname{vol} (\mathcal P) \geq \frac{G}{8} \left( 4 V_{\infty} + V_f - 8\right). \label{eqn8}  
\end{equation}

\section{Proof of the main theorem} \label{sec3}

We now proceed to the proof of Theorem~\ref{theorem1.1}. Let $\mathcal P$ be a right-angled polyhedron of finite volume in $\mathbb H^3$. Denote by $V_{\infty} \geq 0$ its number of ideal vertices and by $V_f \geq 0$ its number of finite vertices. We determine for which $V_{\infty}$ and $V_f$ the inequality $\operatorname{vol} (P) \leq G$ can hold.  

\begin{lemma} \label{lemma3.1}
Suppose one of the following holds for $\mathcal P$: (1) $V_f = 0$; (2) $V_{\infty} = 0$; (3) $V_{\infty}=1$. Then $\operatorname{vol} (\mathcal P) > G$.
\end{lemma}

\begin{proof}
(1) In this case, $\mathcal P$ is an ideal right-angled polyhedron, and by Theorem~\ref{th5},  
$\operatorname{vol} (\mathcal P) \geq v_{oct} = 4 G > G$. 

(2) Here, $\mathcal P$ is a compact right-angled polyhedron, and by Theorem~\ref{th3}, its volume is bounded below by the volume of the right-angled dodecahedron, so $\operatorname{vol} (\mathcal P) \geq 4.306207 > G$.

(3) As shown by Nonaka~\cite[Lemma.~3.1]{Non}, in this case $F \geq 12$, and from equality (\ref{eqn3}), it follows that $V_f \geq 18$. Then, by formula (\ref{eqn8}),  
$\operatorname{vol}(\mathcal P) \geq \frac{G}{8} (4 \cdot 1 + 18 - 8) = \frac{14G}{8} > G$. 
\end{proof}

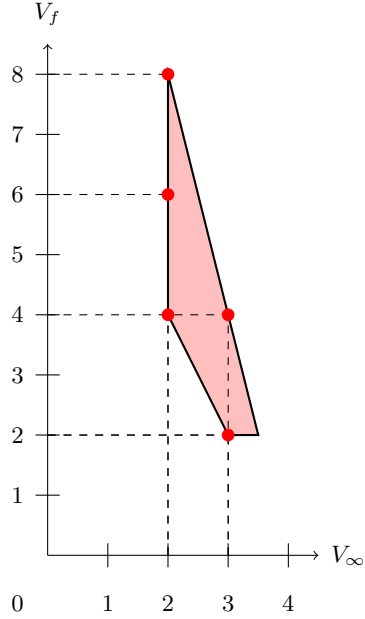
\begin{figure}[ht]	
	\begin{center} 		
\begin{tikzpicture}[scale=0.8] 				
\coordinate (O) at (0,0);
\coordinate (X) at (4.5,0);
\coordinate (Y) at (0,8.5);
%
\draw[black, ->] (O) -- (X); 
\draw[black, ->] (O) -- (Y); 
\draw[black] (1,-0.2) -- (1,0.2);
\draw[black] (2,-0.2) -- (2,0.2);
\draw[black] (3,-0.2) -- (3,0.2);
\draw[black] (4,-0.2) -- (4,0.2);
\draw[black] (-0.2,1) -- (0.2,1);
\draw[black] (-0.2,2) -- (0.2,2);
\draw[black] (-0.2,3) -- (0.2,3);
\draw[black] (-0.2,4) -- (0.2,4);
\draw[black] (-0.2,5) -- (0.2,5);
\draw[black] (-0.2,6) -- (0.2,6);
\draw[black] (-0.2,7) -- (0.2,7);
\draw[black] (-0.2,8) -- (0.2,8);
\draw[black, dashed] (2,0) -- (2,4);
\draw[black, dashed] (3,0) -- (3,2);
\draw[black, dashed] (0,2) -- (3,2);
\node[] at (-0.5,-0.8) {\small $0$};
\node[] at (1,-0.8) {\small $1$};
\node[] at (2,-0.8) {\small $2$};
\node[] at (3,-0.8) {\small $3$};
\node[] at (4,-0.8) {\small $4$};
\node[] at (-0.5,1) {\small $1$};
\node[] at (-0.5,2) {\small $2$};
\node[] at (-0.5,3) {\small $3$};
\node[] at (-0.5,4) {\small $4$};
\node[] at (-0.5,5) {\small $5$};
\node[] at (-0.5,6) {\small $6$};
\node[] at (-0.5,7) {\small $7$};
\node[] at (-0.5,8) {\small $8$};
\node[] at (5,0) {\small $V_{\infty}$};
\node[] at (0,9) {\small $V_{f}$};
\fill[pink] (2,4) -- (3,2) -- (3.5,2) -- (2,8) -- cycle; 
\draw[black, thick] (2,4) -- (3,2) -- (3.5,2) -- (2,8) -- cycle;
\draw[black, dashed] (2,0) -- (2,8);
\draw[black, dashed] (3,0) -- (3,4);
\draw[black, dashed] (0,2) -- (3,2);
\draw[black, dashed] (0,4) -- (3,4);
\draw[black, dashed] (0,6) -- (2,6);
\draw[black, dashed] (0,8) -- (2,8);
\fill[red] (3,2) circle(3pt);
\fill[red] (3,4) circle(3pt);
\fill[red] (2,4) circle(3pt);
\fill[red] (2,6) circle(3pt);
\fill[red] (2,8) circle(3pt);
\end{tikzpicture}
\end{center} \caption{The closed region $\Omega$.} \label{fig:region} 	
\end{figure}

\begin{lemma} \label{lemma3.2}
Let $\Omega$ be the closed region bounded by the quadrilateral with vertices $(2,4)$, $(3,2)$, $(3.5,2)$, and $(2,8)$, as shown in Fig.~\ref{fig:region}. If $\mathcal P$ is such that $(V_{\infty}, V_f) \not\in \Omega$, then $\operatorname{vol} (\mathcal P) > G$.   
\end{lemma}

\begin{proof}
By Lemma~\ref{lemma3.1} and the parity of $V_f$, we may assume that if $\operatorname{vol} (\mathcal P) \leq G$, then $\mathcal P$ has $V_{\infty} \geq 2$ ideal and $V_{f} \geq 2$ finite vertices. By Theorem~\ref{theorem2.1}, the quantities $V_{\infty}$ and $V_f$ satisfy inequality~(\ref{eqn4}). By inequality~(\ref{eqn8}), for $\operatorname{vol} (\mathcal P) \leq G$ to hold, $V_{\infty}$ and $V_f$ must satisfy $4 V_{\infty} + V_f \leq 16$. The system of inequalities 
$$
\begin{cases}
V_{\infty} \geq 2, \cr 
V_{f} \geq 2, \cr
V_{\infty} + \frac{1}{2} V_{f} \geq 4, \cr 
4 V_{\infty} + V_{f} \leq 16
\end{cases}
$$
define the closed region $\Omega$ shown in Fig.~\ref{fig:region}.  
\end{proof}

By Lemmas~\ref{lemma3.1} and~\ref{lemma3.2}, the inequality $\operatorname{vol} (\mathcal P) \leq G$ can hold only if $(V_{\infty}, V_f)$is equal to $(2,4)$, $(2,6)$, $(2,8)$, $(3,2)$, or $(3,4)$. We consider each of these cases below. 

For a polyhedron $\mathcal P$ we define the quantity $W(\mathcal P)$ as the total number of vertices across all its faces. Since each ideal vertex in $\mathcal P$ has degree $4$ and each finite vertex has degree $3$, we get 
\begin{equation} 
W(\mathcal P) = W (V_{\infty}, V_f) = 4 V_{\infty}  + 3 V_f. \label{eqn:W}
\end{equation}  

\begin{lemma} \label{lemma3.3}
If the number of ideal vertices in $\mathcal P$ is $V_{\infty}=2$, then $\operatorname{vol} (\mathcal P) > G$.
\end{lemma}

\begin{proof}
By the above arguments, it remains to consider the three cases \linebreak $(V_{\infty}, V_f) \in \{ (2,4), (2,6), (2,8) \}$. 

\smallskip 
\noindent	
\textbf{Case 1:} $(V_{\infty}, V_f) = (2,4)$. From formulas~(\ref{eqn3}) and~(\ref{eqn:W}), we have $F=6$ and $W(\mathcal P) =  20$. 
Let $p_n$, $n \geq 3$, denote the number of $n$-gonal faces in $\mathcal P$. Then $\sum_{n \geq 3} p_n = F= 6$ and $\sum_{n \geq 3} np_n = W(P) =  20$. Note that the number of triangular faces satisfies the inequality $p_3 \leq 2$. Indeed, by Lemma~\ref{rem}, each triangular face must contain two ideal vertices that belong to a common edge. Since $V_{\infty} = 2$, all triangular faces must contain the same edge. Hence, there are at most two such faces, and each of the remaining four faces has at least four vertices. We obtain the estimate 
$W(\mathcal P) \geq 3 \cdot 2 + 4 \cdot 4 = 22$, which contradicts the equality $W(\mathcal P) = 20$. Thus, \textbf{Case 1} is not realized. 

\smallskip
\noindent		
\textbf{Case 2:} $(V_{\infty}, V_f) = (2, 6)$. From formulas~(\ref{eqn3}) and~(\ref{eqn:W}), we obtain $F=7$ and   	
\begin{equation}
	W(\mathcal P) = 4 V_{\infty} + 3 V_f = 26. \label{eqn26}
\end{equation}
Consider all possible positions for two ideal vertices $v_1$ and $v_2$.		

\smallskip
\noindent 
\textbf{Subcase 2.1:} Suppose that $v_1$ and $v_2$ do not lie in the same triangular face. Then, by Lemma~\ref{rem}, $\mathcal P$ has no triangular faces. Hence, each face contains at least $4$ vertices. Therefore, $W(\mathcal P) \geq 4 F = 28$, which contradicts equality (\ref{eqn26}).

\smallskip
\noindent 
\textbf{Subcase 2.2:} Suppose that $v_1$ and $v_2$ lie in the same triangular face (and so are connected by an edge). Then, as in \textbf{Case 1}, $p_3 \leq 2$. By Lemma~\ref{rem}, each quadrilateral face contains at least one ideal vertex. Keeping in mind that $v_1$ and $v_2$ are connected by an edge, we conclude that  the number of faces containing at least one ideal vertex (and so can be triangular or quadrilateral) does not exceed $6$. Consequently, there is at least one face that contains no ideal vertices and has at least $5$ vertices. Hence $W(\mathcal P) \geq 3 p_3  + 4 (6-p_3)  + 5 \cdot 1 = 29 - p_3 \geq 27$, 
which contradicts (\ref{eqn26}). Thus, \textbf{Case 2} is not realized. 

\smallskip 	
\noindent	
\textbf{Case 3:} $(V_{\infty}, V_f) =(2,8)$. From formulas~(\ref{eqn3}) and~(\ref{eqn:W}), we obtain $F=8$ and   
\begin{equation}
	W(\mathcal P) = 4 V_{\infty} + 3 V_f = 32. \label{eqn100}
\end{equation}
Consider all possible arrangements of the two ideal vertices $v_1$ and $v_2$.

\smallskip 
\noindent
\textbf{Subcase 3.1:} Suppose both ideal vertices $v_1$ and $v_2$ lie in a $k$-gonal face $f$, $k \geq 4$, but are not connected by an edge. By Lemma~\ref{rem}, each quadrilateral face must contain at least one ideal vertex. Therefore, apart from face $f$, vertex $v_1$ can be contained in at most three quadrilateral faces. The same holds for vertex $v_2$. Thus, $\mathcal P$ has a $k$-gonal face $f$ and at most six other quadrilateral faces. Consequently, the eighth face of $\mathcal P$ has only finite vertices and number of vertices in this face is at least $5$. Therefore, $W(\mathcal P) \geq k + 4 \cdot 6 + 5 \cdot 1\geq 33$, given $k \geq 4$, which contradicts (\ref{eqn100}). 	

\smallskip
\noindent		
\textbf{Subcase 3.2:} Suppose the ideal vertices $v_1$ and $v_2$ lie in a $k$-gonal face $f$, $k \geq 4$, and are connected by an edge $e$. Then the face $f_1$ adjacent to $f$ along edge $e$ also contains the ideal vertices $v_1$ and $v_2$. By Lemma~\ref{rem}, each quadrilateral face must contain at least one ideal vertex. Therefore, apart from faces $f$ and $f_1$, vertex $v_1$ can be contained in at most two quadrilateral faces. The same holds for vertex $v_2$. Thus, the number of faces in $\mathcal P$ containing at least one ideal vertex does not exceed $6$ (with $f_1$ possibly being triangular). Hence, there are at least two faces whose vertices are all finite, and these faces contain at least five vertices each. Thus,  
$W(\mathcal P) \geq k + 3 + 4 \cdot 4 + 5 \cdot 2 \geq 33$, given $k \geq 4$, which contradicts (\ref{eqn100}).

\smallskip 
\noindent
\textbf{Subcase 3.3:} Suppose the ideal vertices $v_1$ and $v_2$ are connected by an edge $e$, and lie in triangular faces $T_1$ and $T_2$ both. Let $Q_1$, $Q_2$, $Q_3$, $Q_4$ denote the faces adjacent to $T_1$ or $T_2$. Note that the $Q_i$ are quadrilaterals, as shown in Fig.~\ref{figpk}. Indeed, since each $Q_i$ contains at most one ideal vertex, it must have at least four vertices. Suppose at least one $Q_i$ is a $k$-gon, $k \geq 5$. Since the ideal vertices $v_1$ and $v_2$ are connected by edge $e$, the number of faces in $\mathcal P$ containing at least one ideal vertex does not exceed $6$. Consequently, $\mathcal P$ has at least two faces whose vertices are all finite, and each such face contains at least five vertices. Thus, $W(\mathcal P) \geq  3 \cdot  2 +  4 \cdot  3 + k +  5 \cdot 2 \geq 33$, given $k \geq 5$, which contradicts (\ref{eqn100}). Therefore, all $Q_i$ are quadrilaterals.

For $i=1,2,3,4$, let $u_i$ denote the vertex shared by the common edge of $Q_i$ and $Q_{i+1}$ (with indices modulo $4$) that does not lie in $T_1$ or $T_2$. Let us denote by $w_1$ and $w_2$ finite vertices of triangles $T_1$ and $T_2$ respectively, see Fig.~\ref{figpk}
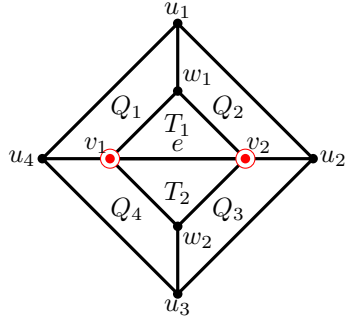
\begin{figure}[ht]	
	\begin{center} 	
		\begin{tikzpicture}[scale=0.9]
			\coordinate (A) at (-1,0);
			\coordinate (B) at (0,1);
			\coordinate (C) at (1, 0); 
			\coordinate (D) at (0, -1);
			\coordinate (E) at (-2, 0);
			\coordinate (F) at (0, 2);
			\coordinate (G) at (2, 0);
			\coordinate (H) at (0, -2);
			\draw[very thick, black] (A) -- (B) -- (C) -- cycle; 
			\draw[very thick, black] (A) -- (D) -- (C) -- cycle; 
			\draw[very thick, black] (E) -- (A);
			\draw[very thick, black] (F) -- (B);
			\draw[very thick, black] (C) -- (G);
			\draw[very thick, black] (D) -- (H);
			\draw[very thick, black] (E) -- (F);
			\draw[very thick, black] (G) -- (F);
			\draw[very thick, black] (H) -- (G);
			\draw[very thick, black] (E) -- (H);
			\foreach \v in {A,C} {
				\fill[white] (\v) circle (4pt); 
				\draw[red] (\v) circle (4pt); 
				\fill[red] (\v) circle (2pt);}			
			\foreach \v in {B,D,F,E,G,H} {
				\fill[black] (\v) circle (2pt); }
				\node[] at (-1.2,0.2) {$v_1$};
				\node[] at (1.2,0.2) {$v_2$};
				\node[] at (0.3,1.2) {$w_1$};
				\node[] at (0.3,-1.2) {$w_2$};
				\node[] at (0.,0.2) {$e$};
		\node[] at (-2.3,0) {$u_4$};
		\node[] at (0,2.2) {$u_1$};
		\node[] at (2.3, 0) {$u_2$};
		\node[] at (0,-2.2) {$u_3$};
		\node[] at (-0.75,0.75) {$Q_1$};
		\node[] at (0.75,0.75) {$Q_2$};
		\node[] at (0.75,-0.75) {$Q_3$};
		\node[] at (-0.75,-0.75) {$Q_4$};
		\node[] at (0,0.5) {$T_1$};
		\node[] at (0,-0.5) {$T_2$};		
		\end{tikzpicture}
	\end{center} \caption{Two adjacent triangles surrounded by quadrilaterals.} \label{figpk} 	
\end{figure}

Consider the following cases. 
\begin{itemize}
	\item[(i)] Suppose the vertices $u_1, u_2, u_3, u_4$ are pairwise distinct. Then all vertices lying in the faces $T_1$, $T_2$, $Q_1$, $Q_2$, $Q_3$, and $Q_4$ have the maximum possible degree, i.e., finite vertices have degree $3$, and ideal vertices have degree $4$. Hence eight vertices, shown in Fig.~\ref{figpk}, are not connected to the remaining two vertices of $\mathcal P$ by edges, that contradicts the connectedness of the 1-skeleton of the polyhedron.  
	\item[(ii)] Suppose two consecutive vertices $u_i$ and $u_{i+1}$ coincide. Then $Q_{i+1}$ collapses into a triangle, contradicting its quadrilateral nature.  
	\item[(iii)] Suppose two non-consecutive vertices $u_i$ and $u_{i+2}$ coincide, while $u_{i+1}$ and $u_{i+3}$ are distinct. If $i \in \{1,3\}$ then vertex $u_1=u_3$ is adjacent to four vertices $w_1$, $w_2$, $u_2$, and $u_4$, contradicting its degree of $3$. Analogously, if $i \in \{2,4\}$ then vertex $u_2=u_4$ is adjacent to four vertices  $w_1$, $w_2$, $u_2$, and $u_4$, contradicting its degree of $3$.  
	\item[(iv)] Suppose the vertices $u_i$ and $u_{i+2}$ coincide, and also $u_{i+1}$ and $u_{i+3}$ coincide. Then all vertices lying in the faces $T_1$, $T_2$, $Q_1$, $Q_2$, $Q_3$, and $Q_4$ have the maximum possible degree, i.e., the finite vertices $\{v_1, v_3, u_i=u_{i+2}, u_{i+1}=u_{i+3}\}$ have degree $3$, and the ideal vertices $\{v_2, v_4 \}$ have degree $4$. Hence, these vertices are not connected to the remaining four vertices of $\mathcal P$ by edges, contradicting the connectedness of the 1-skeleton of the polyhedron. 
\end{itemize}

\smallskip
\noindent
\textbf{Subcase 3.4:} Suppose the ideal vertices $v_1$ and $v_2$ do not lie in a common face, and $\mathcal P$ has a $k$-gonal face, where $k \geq 5$. By Lemma~\ref{rem}, $\mathcal P$ cannot have triangular faces. Hence, $W(\mathcal P) \geq 5 + 7 \cdot 4 = 33$, which contradicts (\ref{eqn100}).

\smallskip 
\noindent
\textbf{Subcase 3.5:} Suppose all eight faces of $\mathcal P$ are quadrilaterals, so each face contains exactly one ideal vertex. 
\begin{figure}[ht]	
	\begin{center} 	
		\begin{tikzpicture}[scale=0.9]
			\coordinate (O) at (0,0);
			\coordinate (A) at (-1,-1);
			\coordinate (B) at (-1,1);
			\coordinate (C) at (1, 1); 
			\coordinate (D) at (1, -1);
			\coordinate (A1) at (-2,-2);
			\coordinate (B1) at (-2,2);
			\coordinate (C1) at (2, 2); 
			\coordinate (D1) at (2, -2);
			\coordinate (A2) at (-1,0);
			\coordinate (B2) at (0,1);
			\coordinate (C2) at (1, 0); 
			\coordinate (D2) at (0, -1);
			\draw[very thick, black] (A) -- (A2) -- (B)-- (B2) -- (C)-- (C2) -- (D)-- (D2) -- cycle; 
			\draw[very thick, black] (O) -- (A2);
			\draw[very thick, black] (O) -- (B2);
			\draw[very thick, black] (O) -- (C2);
			\draw[very thick, black] (O) -- (D2);
			\draw[very thick, black] (A) -- (A1);
			\draw[very thick, black] (B) -- (B1);
			\draw[very thick, black] (C) -- (C1);
			\draw[very thick, black] (D) -- (D1);
			\foreach \v in {O} {
				\fill[white] (\v) circle (4pt);
				\draw[red] (\v) circle (4pt); 
				\fill[red] (\v) circle (2pt);}				
			\foreach \v in {A,B,C,D,A2,B2,C2,D2} {
				\fill[black] (\v) circle (2pt);}
			\node[] at (-0.6,-0.6) {$Q_1$};
			\node[] at (-0.6, 0.6) {$Q_2$};
			\node[] at (0.6,0.6) {$Q_3$};
			\node[] at (0.6,-0.6) {$Q_4$};
			\node[] at (0.25,0.25) {$v_1$};
			\node[] at (-1.25,0) {$q_1$};
			\node[] at (0, 1.25) {$q_2$};
			\node[] at (1.25,0) {$q_3$};
			\node[] at (0,-1.25) {$q_4$};
			\node[] at (-2,0) {$Q'_1$};
			\node[] at (0,2) {$Q'_2$};
			\node[] at (2,0) {$Q'_3$};
			\node[] at (0,-2) {$Q'_4$};
			\node[] at (0,-2.7) {$(a)$};
		\end{tikzpicture}
\qquad \qquad 
\begin{tikzpicture}[scale=0.9] 
\coordinate (A) at (0,0);
\coordinate (B) at (4,0);
\coordinate (C) at (2, 3.5); 
\coordinate (D) at (1.5, 2);
\coordinate (E) at (2, 2.2);
\coordinate (F) at (2.5, 2);
\coordinate (G) at (2, 1);
\coordinate (H) at (1, 1);
\coordinate (I) at (3, 1);
\coordinate (J) at (2, 0);
\draw[very thick, black] (A) -- (B) -- (C) -- cycle; 
\draw[very thick, black] (E) -- (D);
\draw[very thick, black] (E) -- (F);
\draw[very thick, black] (D) -- (C);
\draw[very thick, black] (F) -- (C);
\draw[very thick, black] (E) -- (G);
\draw[very thick, black] (G) -- (H);
\draw[very thick, black] (G) -- (I);
\draw[very thick, black] (G) -- (J);
\draw[very thick, black] (H) -- (A);
\draw[very thick, black] (I) -- (B);
\draw[very thick, black] (H) -- (D);
\draw[very thick, black] (I) -- (F);		
\foreach \v in {C,G} {
\fill[white] (\v) circle (4pt); 
\draw[red] (\v) circle (4pt); 
\fill[red] (\v) circle (2pt);}			
\foreach \v in {B,A,D,E,F,H,I,J} {
\fill[black] (\v) circle (2pt); }
\fill[white] (0,-0.5) circle (2pt);
\node[] at (2.3,1.2) {$v_1$}; 
\node[] at (2.3,3.8) {$v_2$}; 
	\node[] at (2,-0.7) {$(b)$};
\end{tikzpicture}		
\end{center} \caption{The polyhedron $P_{(2,8)}$ and its Schlegel diagram.} \label{fig28} 	
\end{figure}
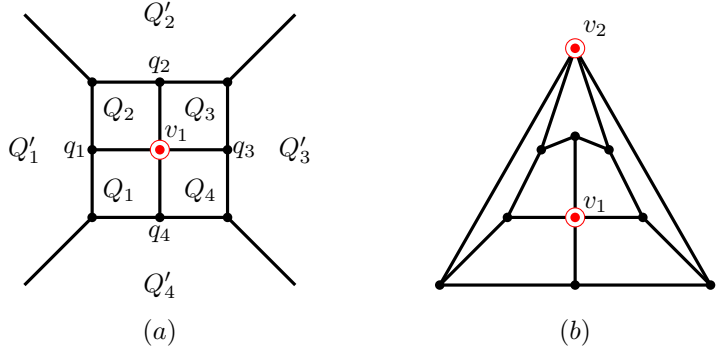 
Denote the ideal vertices in $\mathcal P$ by $v_1$ and $v_2$. 
Let $Q_1$, $Q_2$, $Q_3$, $Q_4$ be the quadrilateral faces containing $v_1$, as shown in Fig.~\ref{fig28}\,(a). Let $q_i$, $i= 1, \dots, 4$, be the finite vertex shared by $Q_i$ and $Q_{i+1}$ (with indices modulo $4$). Let $Q'_i$ be the quadrilateral face sharing vertex $q_i$ with $Q_i$ and $Q_{i+1}$. Since $\mathcal P$ has 16 edges, the four edges where $Q'_i$ and $Q'_{i+1}$ intersect must meet at the ideal vertex $v'$, which is assumed to be far away in picture Fig.~\ref{fig28}\,(a). Thus, this case corresponds to a unique polyhedron. A Schlegel diagram of the same polyhedron is shown in Fig.~\ref{fig28}\,(b).  
Since this polyhedron has $V_{\infty}=2$ and $V_f=8$, we denote it by $P_{(2,8)}$. 
In Fig.~\ref{fig28new} we give a presentation of $P_{(2,8)}$, where the left and right edges are assumed to be identified along $AB_1C_1 D$. The figure shows that $P_{(2,8)}$ has a dihedral symmetry group of order eight, generated by reflections in the planes $(AC_3D)$ and $(AB_3D)$, intersecting along the line $AD$. 
\begin{figure}[ht]	
	\begin{center} 	
\begin{tikzpicture}[scale=1.0]
\coordinate (A) at (0,1);
\coordinate (B) at (1,2);
\coordinate (C) at (2,1); 
\coordinate (D) at (3, 2);
\coordinate (E) at (4, 1);
\coordinate (F) at (5, 2);
\coordinate (G) at (6, 1);
\coordinate (H) at (7, 2);
\coordinate (I) at (8, 1);
\coordinate (J) at (9, 2);
\coordinate (X) at (4.5, -0.5);
\coordinate (Y) at (4.5, 3.3);
\fill[cyan] (Y) -- (F) -- (X) -- cycle; 	
\fill[pink] (Y) -- (E) -- (X) -- cycle; 			
\draw[very thick, black] (A) -- (B) -- (C) -- (D) -- (E) -- (F) -- (G) -- (H) -- (I) -- (J); 	
\draw[very thick, dashed, black] (Y) -- (E); 
\draw[very thick, dashed, black] (X) -- (F); 
\draw[very thick, dotted, black] (X) -- (Y); 						
\foreach \v in {A,B,C,D,E,G,F,H,I,J} {
\fill[black] (\v) circle (2pt); }
\foreach \v in {A,C,E,G,I} {
\draw[very thick, black] (X) -- (\v); }
\foreach \v in {B,D,F,H,J} {
\draw[very thick, black] (Y) -- (\v); }
\foreach \v in {X,Y} {
\fill[white] (\v) circle (4pt); 
\draw[red] (\v) circle (4pt); 
\fill[red] (\v) circle (2pt);}
\node[] at (4.5, 3.8) {$A$};
\node[] at (4.5,-1) {$D$};
\node[] at (0.7,2.2) {$B_1$};
\node[] at (2.7,2.2) {$B_2$};
\node[] at (5.3,2.2) {$B_3$};
\node[] at (7.3,2.2) {$B_4$};
\node[] at (9.3,2.2) {$B_1$};
\node[] at (-0.3,0.8) {$C_1$};
\node[] at (1.7,0.8) {$C_2$};
\node[] at (3.7,0.8) {$C_3$};
\node[] at (6.3,0.8) {$C_4$};
\node[] at (8.3,0.8) {$C_1$};
\end{tikzpicture}	
	\end{center} \caption{The polyhedron $P_{(2,8)}$ and the tetrahedron $\Delta_{4,4,4}$.} 
	\label{fig28new} 	
\end{figure}
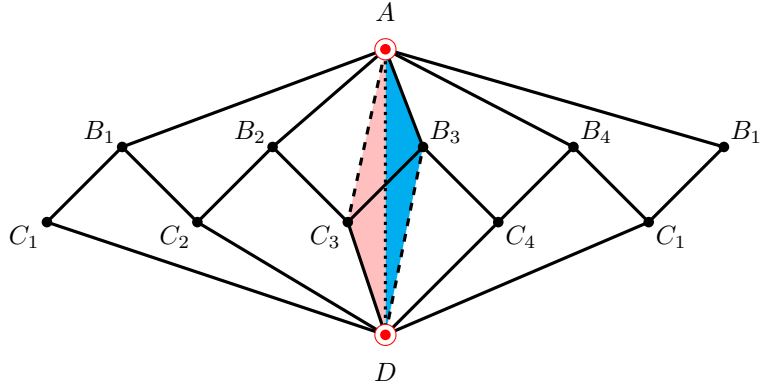	
When $P_{(2,8)}$ is quotient by this dihedral symmetry group, we obtain the tetrahedron $\Delta_{4,4,4} = ADB_3C_3$, whose dihedral angles at edges $AD$, $AB_3$, and $C_3D$ are $\pi/4$, and the remaining angles are $\pi/2$. The Coxeter diagram of the group $\Gamma(\Delta_{4,4,4})$ is shown in Fig.~\ref{figCo2}.	
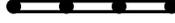
\begin{figure}[ht]	
	\begin{center} 
		\begin{tikzpicture}[scale=0.7] 				
			\draw[very thick, black] (0,-0.1) -- (1,-0.1);
			\draw[very thick, black] (0,0.1) -- (1,0.1);
			\draw[very thick, black] (1,-0.1) -- (2,-0.1);
			\draw[very thick, black] (1,0.1) -- (2,0.1);
			\draw[very thick, black] (2,-0.1) -- (3,-0.1);
			\draw[very thick, black] (2,0.1) -- (3,0.1);
			\fill[black] (0,0) circle (4pt); 
			\fill[black] (1,0) circle (4pt);
			\fill[black] (2,0) circle (4pt);
			\fill[black] (3,0) circle (4pt); 			
		\end{tikzpicture}
	\end{center} 
	\caption{Coxeter diagram of the group $\Gamma(\Delta_{4,4,4})$.
	} \label{figCo2} 	
\end{figure}	
Since $\operatorname{vol} (\Delta_{4,4,4}) = \frac{1}{2} \Lambda(\frac{\pi}{4})$, we have $\operatorname{vol} (P_{(2,8)}) = 4 \Lambda (\frac{\pi}{4}) = 2G > G$. The proof of Lemma~\ref{lemma3.3} is completed. 
\end{proof}

\begin{lemma} \label{lemma3.4}
If the number of ideal vertices in $\mathcal P$ is $V_{\infty}=3$, then $\operatorname{vol} (\mathcal P) \geq G$. Moreover, equality holds only if $\mathcal P$ is a right-angled  triangular bipyramid $\mathcal P_{(3,2)}$. 
\end{lemma}

\begin{proof}
By Lemma~\ref{lemma3.2}, it remains to consider two cases: $(V_{\infty}, V_f) \in \{ (3,2), (3,4) \}$. To follow the general enumeration of cases belonging to the regioin $\Omega$, we will refer to these as the fourth and fifth cases.

\noindent
\textbf{Case 4:} $(V_\infty, V_f) = (3,2)$. From formulas~(\ref{eqn3}) and~(\ref{eqn:W}), we obtain $F=6$ and $W(P) = 18$. Assume that $\mathcal P$ has at least one face with four or more vertices. Then $W(\mathcal P) \geq 4 + 5 \cdot 3 = 19$, which leads to a contradiction. Hence, all faces of $\mathcal P$ must be triangles. Let us denote ideal vertices of $\mathcal P$ by $v_1$, $v_2$ and $v_3$. Since each of the six triangular faces must contain at least two ideal vertices, $\mathcal P$ must have at least three edges connecting vertices $v_1$, $v_2$ and $v_3$. Thus, the three ideal vertices form a cycle of length three in the 1-dimensional skeleton of the polyhedron, and the six triangular faces adjoin the edges of this cycle---two faces per edge. Hence $\mathcal P$ coincides with the polyhedron $P_{(3,2)}$ shown in Fig.~\ref{fig32}\,(b), whose volume is $\operatorname{vol}(P_{(3,2)}) = 2 \Lambda \left( \frac{\pi}{4} \right) = G$. 

\smallskip 
\noindent
\textbf{Case 5:} $(V_\infty, V_f) = (3,4)$. From formulas~(\ref{eqn3}) and~(\ref{eqn:W}), we obtain $F=7$ and  
\begin{equation} 
	W(\mathcal P) = 3 V_f + 4 V_{\infty} = 24. \label{eqn200}
\end{equation}
First, note that $\mathcal P$ has only triangular and quadrilateral faces. Indeed, assume for contradiction that there exists a $k$-gonal face $f$ with $k \geq 5$. Observe that $f$ must contain at least two ideal vertices. Otherwise, $f$ would have at least three edges with only finite vertices. Then the three faces adjacent to $f$ along these edges would have at least four vertices each. Hence, $W(P) \geq k + 4 \cdot 3 + 3 \cdot 3 \geq 26$, since $k \geq 5$, which leads to a contradiction. But if $f$ has at least two ideal vertices, then $f$ is adjacent to at least seven faces, meaning $\mathcal P$ has at least eight faces. This contradicts $F=7$. Thus, $\mathcal P$ contains only triangular and quadrilateral faces. Moreover, since $W(\mathcal P) = 3 p_3 + 4 p_4 = 24$ and $p_3 + p_4 = 7$, it follows that $p_3 = 4$ and $p_4 = 3$. 

Let us consider all possible cases of positions of triangular and quadrilateral faces.  

\smallskip 
\noindent
\textbf{Subcase 5.1:} Suppose $\mathcal P$ has a triangular face $T_0$ containing all three ideal vertices $v_1$, $v_2$ and $v_3$. Then the remaining triangular faces $T_i$, $i=1, 2, 3$, are adjacent to $T_0$ along edges, as shown in Fig.~\ref{fig34}\,(a). For $i=1,2,3$, let $Q_i$ denote the quadrilateral face sharing an ideal vertex with triangles $T_0$, $T_i$, and $T_{i+1}$. 
\begin{figure}[ht]	
	\begin{center} 	
\begin{tikzpicture}[scale=1.]
	\coordinate (A) at (0.25,0.25);
	\coordinate (B) at (2.75,0.25);
	\coordinate (C) at (1.5, 2.25); 
	\coordinate (D1) at (1.5, -1);
	\coordinate (D2) at (-0.5, 2);
	\coordinate (D3) at (3.4, 2);
	\coordinate (E) at (0.25, 1.75);
	\coordinate (F) at (2.8, 1.75);
	\coordinate (G) at (1.5, -0.6);
	\draw[very thick, black] (A) -- (B) -- (C) -- cycle; 
	\draw[very thick, black] (D2) -- (E);
	\draw[very thick, black] (D3) -- (F);
	\draw[very thick, black] (C) -- (F);
	\draw[very thick, black] (C) -- (E);
	\draw[very thick, black] (A) -- (G);
	\draw[very thick, black] (B) -- (G);
	\draw[very thick, black] (D1) -- (G);
	\draw[very thick, black] (E) -- (A);
	\draw[very thick, black] (F) -- (B);
	\foreach \v in {A,B,C} {
		\fill[white] (\v) circle(4pt);
		\draw[red] (\v) circle (4pt); 
		\fill[red] (\v) circle (2pt);}				
	\foreach \v in {E,F,G} {
		\fill[black] (\v) circle (2pt);}	
		\node[] at (-0.3,0.1) {$Q_1$};
		\node[] at (1.5,2.75) {$Q_2$};
		\node[] at (3.3,0.1) {$Q_3$};
		\node[] at (1.5,0.95) {$T_0$};
		\node[] at (1.5,-0.1) {$T_1$};
		\node[] at (0.5,1.3) {$T_2$};
		\node[] at (2.5,1.3) {$T_3$};
		\node[] at (1.5,-1.5) {(a)};
\end{tikzpicture}
\qquad \qquad 
\begin{tikzpicture}[scale=1.1]
\coordinate (A) at (0,0);
\coordinate (B) at (3,0);
\coordinate (C) at (1.5, 2.5); 
\coordinate (D) at (1.5, 0.9);
\coordinate (E) at (1.25, 1.2);
\coordinate (F) at (1.75, 1.2);
\coordinate (G) at (1.5, 0.6);
\draw[very thick, black] (A) -- (B) -- (C) -- cycle; 
\draw[very thick, black] (D) -- (E);
\draw[very thick, black] (D) -- (F);
\draw[very thick, black] (C) -- (F);
\draw[very thick, black] (C) -- (E);
\draw[very thick, black] (A) -- (G);
\draw[very thick, black] (B) -- (G);
\draw[very thick, black] (D) -- (G);
\draw[very thick, black] (E) -- (A);
\draw[very thick, black] (F) -- (B);
\foreach \v in {A,B,C} {
\fill[white] (\v) circle(4pt);
\draw[red] (\v) circle (4pt); 
\fill[red] (\v) circle (2pt);}				
\foreach \v in {D,E,F,G} {
\fill[black] (\v) circle (2pt);}
\fill[white] (0,-0.5) circle (2pt);
		\node[] at (1.5,-1.) {(b)};
\end{tikzpicture}
	\end{center} \caption{The polyhedron $P_{(3,4)}$ and its Schlegel diagram.} \label{fig34} 	
\end{figure}
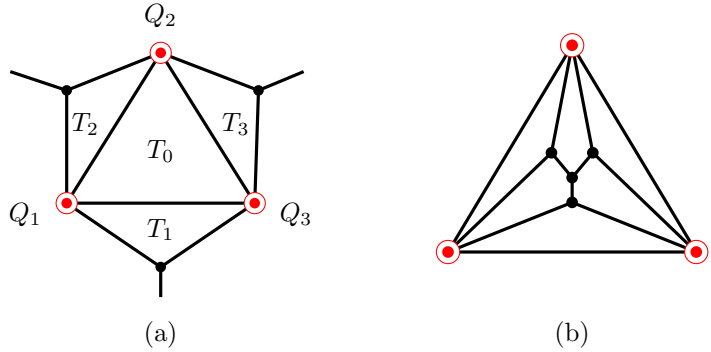
Since $\mathcal P$ has 12 edges, the three edges, along which pairs of faces $Q_i Q_{i+1}$ intersect, must share a common finite vertex. In Fig.~\ref{fig34}\,(a), this vertex is assumed to be pictured far away. The Schlegel diagram of the same polyhedron is presentd in Fig.~\ref{fig34}\,(b). Since this polyhedron has $3$ ideal and $4$ finite vertices, we will denote it by $P_{(3,4)}$. 

To find the volume $\operatorname{vol} (\mathcal P_{(3,4)})$, note that under the action of the dihedral group of order 4 generated by reflections in the faces of $\mathcal P_{(3,4)}$ passing through the finite vertices $A$, $B$, $D$, and the finite vertices $A$, $C$, $D$, we obtain the rectangular quadrilateral antiprism $\mathcal A_4$, see Fig.~\ref{fig:a4p34}. Thus, $\operatorname{vol}(P_{(3,4)}) = \frac{1}{4} \operatorname{vol}(\mathcal A_4)$. Using formula~(\ref{eqnantiprism}), we obtain the approximate value $\operatorname{vol}(P_{(3,4)}) = 1.505361 > G$.
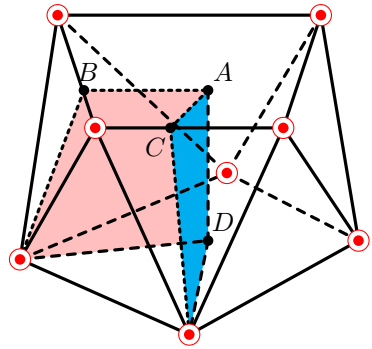
\begin{figure}[ht]	
	\begin{center} 
		\begin{tikzpicture}[scale=1.0, line join=round, line cap=round]
		\coordinate (0) at (-14.75, 0);
		\coordinate  (1) at (-17, 1);
		\coordinate (2) at (-12.5, 1.25);
		\coordinate (3) at (-14.25, 2.15);
		\coordinate  (4) at (-16.5, 4.25);
		\coordinate (5) at (-13, 4.25);
		\coordinate (6) at (-13.5, 2.75);
		\coordinate (7) at (-16, 2.75);
		\coordinate (8) at (-14.5, 1.25); 
		\coordinate (9) at (-15, 2.75);
		\coordinate (10) at (-14.5, 3.25); 
		\coordinate (11) at (-16.15, 3.25);
		\fill[pink] (10) -- (8)  -- (1) -- (11) -- cycle; 
		\fill[cyan] (10) -- (9) -- (0) -- (8) -- cycle; 
		\draw[very thick, black, dashed] (1) to (3);
		\draw[very thick, black, dashed] (3) to (2);
		\draw[very thick, black] (2) to (0);
		\draw[very thick, black] (0) to (1);
		\draw[very thick, black] (4) to (1);
		\draw[very thick, black, dashed] (4) to (3);
		\draw[very thick, black, dashed] (5) to (3);
		\draw[very thick, black] (5) to (2);
		\draw[very thick, black] (0) to (6);
		\draw[very thick, black] (6) to (2);
		\draw[very thick, black] (6) to (5);
		\draw[very thick, black] (5) to (4);
		\draw[very thick, black] (4) to (7);
		\draw[very thick, black] (7) to (6);
		\draw[very thick, black] (7) to (0);
		\draw[very thick, black] (7) to (1);
		\draw[very thick, black, dashed] (1) to (8);
		\draw[very thick, black, dashed] (8) to (0);
		\draw[very thick, black, dotted] (10) to (9);
		\draw[very thick, black, dashed] (10) to (8);
		\draw[very thick, black, dotted] (11) to (10);
		\draw[very thick, black, dotted] (11) to (1);
		\draw[very thick, black, dotted] (9) to (0);
			\foreach \v in {0,1,2,3,4,5,6,7} {
			\fill[white] (\v) circle (4pt); 
			\draw[red] (\v) circle (4pt); 
			\fill[red] (\v) circle (2pt);}				
		\foreach \v in {8,9,10,11} {
			\fill[black] (\v) circle (2pt);}
			\node[] at (-14.3, 3.5) {$A$};
			\node[] at (-16.1, 3.5) {$B$};
			\node[] at (-15.2, 2.5) {$C$};
			\node[] at (-14.3, 1.5) {$D$}; 
		\end{tikzpicture}
	\end{center} \caption{The polyhedron $P_{(3,4)}$ as a $\frac{1}{4}$-slice of the antiprism $\mathcal A_4$.} \label{fig:a4p34} 	
\end{figure}

Next, we assume that $\mathcal P$ has no triangular face containing all three ideal vertices. Recall that in the case under consideration, $\mathcal P$ has only triangular and quadrilateral faces. 

\smallskip 
\noindent
\textbf{Subcase 5.2:} Suppose $\mathcal P$ has a quadrilateral face $Q_1$ with exactly one ideal vertex, say $v_1$. We follow the notation in Fig.~\ref{fig41}\,(a). 
\begin{figure}[ht]	
	\begin{center} 	
		\begin{tikzpicture}[scale=0.9]
			\coordinate (A) at (-1,-1);
			\coordinate (B) at (-1,1);
			\coordinate (C) at (1, 1); 
			\coordinate (D) at (1, -1);
			\coordinate (E1) at (-0.5, 1.75);
			\coordinate (E2) at (0.5, 1.75);
			\coordinate (F1) at (1.75, 0.5);
			\coordinate (F2) at (1.75, -0.5);
			\coordinate (G) at (-1.5, -1.5); 		
			\draw[very thick, black] (A) -- (B) -- (C) -- (D) -- cycle; 
			\draw[very thick, black] (B) -- (E1);
			\draw[very thick, black] (C) -- (E2);
			\draw[very thick, black] (C) -- (F1);
			\draw[very thick, black] (D) -- (F2);
			\draw[very thick, black] (A) -- (G);
			\foreach \v in {C} {
				\fill[white] (\v) circle (4pt); 
				\draw[red] (\v) circle (4pt); 
				\fill[red] (\v) circle (2pt);}				
			\foreach \v in {A,B,D} {
				\fill[black] (\v) circle (2pt);}
			\node[] at (0,0) {$Q$};
			\node[] at (-1.5,0) {$Q_2$};
			\node[] at (0,-1.5) {$Q_1$};
			\node[] at (1.5,0) {$T_2$};
			\node[] at (0,1.5) {$T_1$};
			\node[] at (1.3,1.3) {$T_3$};	
			\node[] at (0.8,0.8) {$v$};
		\end{tikzpicture}
		\qquad \qquad 
		\begin{tikzpicture}[scale=0.9] 
			\coordinate (A) at (-1,-1);
			\coordinate (B) at (-1,1);
			\coordinate (C) at (1, 1); 
			\coordinate (D) at (1, -1);
			\coordinate (E) at (0, 2.75);
			\coordinate (F) at (2.75, 0);
			\coordinate (G) at (-1.5, -1.5);				
			\draw[very thick, black] (A) -- (B) -- (C) -- (D) -- cycle; 
			\draw[very thick, black] (B) -- (E) -- (C) -- cycle;
			\draw[very thick, black] (C) -- (F) -- (D) -- cycle;
			\draw[very thick, black] (A) -- (G);
			\foreach \v in {C,E,F} {
				\fill[white] (\v) circle (4pt); 
				\draw[red] (\v) circle (4pt); 
				\fill[red] (\v) circle (2pt);}				
			\foreach \v in {A,B,D} {
				\fill[black] (\v) circle (2pt);}
			\node[] at (0,0) {$Q$};
			\node[] at (-1.5,0) {$Q_2$};
			\node[] at (0,-1.5) {$Q_1$};
			\node[] at (1.5,0) {$T_2$};
			\node[] at (0,1.5) {$T_1$};
			\node[] at (1.3,1.3) {$T_3$}; 	
			\node[] at (0.8,0.8) {$v$};
			\node[] at (0,3.07) {$v_1$};
			\node[] at (3.07,0) {$v_2$};
		\end{tikzpicture}
	\end{center} \caption{\textbf{Subcase 5.2}: Quadrilateral $Q_1$ with one ideal vertex $v_1$.} \label{fig41} 	
\end{figure}
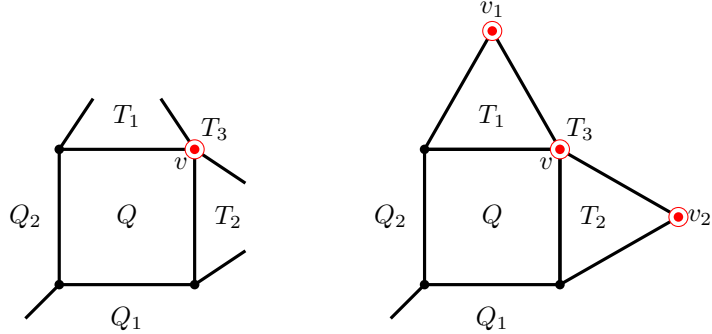
Since the faces $Q_1$ and $Q_2$ contain edges with both vertices finite, they cannot be triangles and must therefore be quadrilaterals. Since the number of quadrilateral faces is three, $\{ Q_1, Q_2, Q_3\}$ is the complete list of  quadrilateral faces in $P$, and the remaining 4 faces are triangular. Let us denote of them by $T_1$, $T_2$, and $T_3$, as in Fig.~\ref{fig41}\,(a). Then $T_1$ contains an ideal vertex $v_2$ adjacent to $v_1$, and $T_2$ contains an ideal vertex $v_3$ adjacent to $v_1$. Since $T_3$ is a triangle, the vertices $v_2$ and $v_3$ must be connected by an edge. Thus, $T_3$ is a triangular face of $\mathcal P$ containing all three ideal vertices $v_1$, $v_2$, $v_2$. This situation has already been addressed in~\textbf{Subcase~5.1}. 

\smallskip 
\noindent
\textbf{Subcase 5.3:} Suppose that $\mathcal P$ has no quadrilateral face with exactly one ideal vertex. By Lemma~\ref{rem}, each quadrilateral face must contain at least one ideal vertex, so every quadrilateral face of $\mathcal P$ has at least two ideal vertices. Let $Q_1$, $Q_2$, and $Q_3$ denote the quadrilateral faces of $\mathcal P$, and let $k_1 \geq 2$, $k_2 \geq 2$, and $k_3 \geq 2$ be the numbers of ideal vertices in them respectively. For a polyhedron $\mathcal P$, define by $WI(\mathcal P)$ the sum of number of ideal vertices counted across all faces. Since $V_{\infty}=3$ and each ideal vertex has degree $4$, we have $WI(\mathcal P)=4 \cdot 3 = 12$. On the other hand, since $\mathcal P$ also has $Q_1, Q_2, Q_3$ and four triangular faces, each containing exactly two ideal vertices, the total number of ideal vertices across all faces is $WI(\mathcal P)=2 \cdot 4 + k_1 + k_2 + k_3 \geq 8 + 6 = 14$. This leads to a contradiction. 
\end{proof}

It follows rom Lemmas~\ref{lemma3.1},~\ref{lemma3.2},~\ref{lemma3.3}, and~\ref{lemma3.4},  that the volume of any right-angled hyperbolic polyhedron is bounded below by Catalan's constant $G$, with equality achieved only for the polyhedron $P_{(3,2)}$. This completes the proof of Theorem~\ref{theorem1.1}.

\section{Arithmeticity of Right-Angled Coxeter Groups}  \label{sec4}

It is well known that the arithmeticity of discrete groups $\Gamma < \operatorname{Isom} (\mathbb H^3)$ of finite covolume play an important role in studying of hyperbolic manifolds and orbifolds $\mathbb H^3 / \Gamma$, see~\cite{MaRe}. Here we only mention about the following important property: by Margulis's theorem, see, e.g.,~\cite[Th.~10.3.5]{MaRe}, the commensurator  
$$
\operatorname{Comm (\Gamma)} = \{ \gamma \in \operatorname{Isom} (\mathbb H^3) \, | \, \gamma \Gamma \gamma^{-1} \text{ and } \Gamma \, \text{are commensurable}\}
$$ 
is discrete if and only if $\Gamma$ is non-arithmetic. 

In~\cite{Vi1}, Vinberg established necessary and sufficient conditions for arithmeticity of discrete groups of motions of $\mathbb H^n$ generated by finitely many reflections and having a fundamental polyhedron of finite volume. Since right-angled Coxeter groups are under discussion in this paper, we recall that it was noted in~\cite{Ve3} that the group  $\Gamma(\mathcal L_n)$, generated by reflections in faces of a compact right-angled L{\"o}bell polyhedron, is non-arithmetic for $n \not\in\{5,6,7,8,10,12,18\}$, and later in~\cite{AMR}, it was shown that $\Gamma(\mathcal L_n)$ is arithmetic if and only if $n\in\{5,6,8\}$ (see also~\cite{BoDu}). Using Vinberg's arithmeticity criterion, it was established in~\cite{Ke3} that the group $\Gamma(\mathcal A_n)$, generated in faces of an ideal right-angled antiprism, is arithmetic if and only if $n\in\{3, 4\}$. 

As noted in~\cite{Vi1}, Vinberg's arithmeticity conditions simplify significantly if the fundamental polyhedron $P$ of $\Gamma(P)$ is non-compact. Namely, let $A(P) = (a_{ij})_{i,j=1}^N$ be the Gram matrix of $P$. Let $\operatorname{Cyc}(A)$ denote the set of all cyclic products of the form $a_{i_1 i_2} a_{i_2 i_3} \cdots a_{i_{m-1} i_m} a_{i_m i_1}$. Then, for $\Gamma$ to be arithmetic, it is necessary and sufficient that all cyclic products in $\operatorname{Cyc}(2 \cdot A(P))$ lie in $\mathbb Z$. For example, for the group $\Gamma(\Delta_{4,4,4})$ the  doubled Gram matrix has the form 
$$
2 \cdot A (\Delta_{4,4,4}) = 
\begin{pmatrix} 
2 & - \sqrt{2} & 0 & 0 \cr 
-\sqrt{2} & 2 & -\sqrt{2} & 0 \cr 
0 & -\sqrt{2} & 2 & -\sqrt{2}  \cr 
0 & 0 & -\sqrt{2} & 2
\end{pmatrix}, 
$$
that admits to verify arithmeticity of this group easily.  

It is known that $\Gamma(\Delta_{3,4,4})$ and $\Gamma(\Delta_{4,4,4})$ are commensurable with the Picard group $\operatorname{PSL} (2,\mathbb Z{\sqrt{-1}})$, see, e.g.,~\cite[Fig.~13.3]{MaRe}. Since $\Gamma (P_{(3,2)})$ is commensurable with $\Gamma(\Delta_{3,4,4})$ (see Fig.~\ref{fig32}), $\Gamma(P_{(2,8)})$ is commensurable with $\Gamma(\Delta_{4,4,4})$ (see Fig.~\ref{fig28new}), and $\Gamma (P_{(3,4)})$ is commensurable with the group generated by reflections in the faces of the right-angled antiprism $\mathcal A_4$ (see Fig.~\ref{fig:a4p34}), all three groups $\Gamma (P_{(3,2)})$, $\Gamma (P_{(2,8)})$, and $\Gamma (P_{(3,2)})$ are arithmetic. Let us fix this face as the following remark. 

\begin{remark}
The right-angled hyperbolic Coxeter groups $\Gamma (P_{(3,2)})$, $\Gamma (P_{(3,4)})$, and $\Gamma(P_{(2,8)})$ are arithmetic. 
\end{remark} 

\section{Open Questions} \label{sec5}

We conclude by formulating some open questions. 

\begin{question}
Classify arithmetic right-angled hyperbolic Coxeter groups.
\end{question}

In~\cite[Prop.~5]{Kol}, it was shown that the antiprism $\mathcal A_n$, $n \geq 3$, is minimal in terms of the number of faces among ideal hyperbolic polyhedra with an $n$-gonal face. A natural question arises about polyhedra with a similar property in the class of compact polyhedra. 

\begin{question}
Is it true that the L{\"o}bell polyhedron $L(n)$, $n \geq 5$, has minimal number of faces is the class of compact right-angled hyperbolic polyhedra having at least one $n$-gonal face?
\end{question}	

In~\cite{DrKe}, the minimal co-volume non-arithmetic hyperbolic Coxeter group with a non-compact fundamental polyhedron was found. A natural question arises about a right-angled Coxeter group with a similar property. 

\begin{question} 
What is a minimal co-volume non-arithmetic hyperbolic Coxeter group? 
\end{question}

We also recall a question posed in~\cite[p.~66]{PoVi}. 

\begin{question}
Is it true that the smallest number of hyperfaces in a compact right-angled polyhedron in $\mathbb H^4$ is $120$?
\end{question}

\end{document}